\documentclass{amsart}
\usepackage{amsmath,amssymb,amsfonts,amsthm}
\usepackage[dvips]{graphicx}
\usepackage{setspace}
\usepackage{url}
\usepackage{hyperref}
\usepackage{enumerate}

\begin{document}
\newcommand{\sidebar}[1]{\vskip10pt\noindent
 \hskip.70truein\vrule width2.0pt\hskip.5em
 \vbox{\hsize= 4truein\noindent\footnotesize\relax #1 }\vskip10pt\noindent}

\newcommand{\omegaone}{\ensuremath{\omega_1}}
\newcommand{\lomegaone}{\ensuremath{\mathcal{L}_{\omega_1,\omega}}}
\newcommand{\alephs}[1]{\ensuremath{\aleph_{#1}}}
\newcommand{\alephalpha}{\alephs{\alpha}}
\newcommand{\alephomegaone}{\alephs{\omegaone}}
\newcommand{\alephalphaplus}{\alephs{\alpha+1}}
\newcommand{\beths}[1]{\ensuremath{\beth_{#1}}}
\newcommand{\bethalpha}{\beths{\alpha}}
\newcommand{\bethomegaone}{\beths{\omegaone}}
\newcommand{\bethalphaplus}{\beths{\alpha+1}}
\newcommand{\z}{\ensuremath{\mathcal{Z}}}
\newcommand{\M}{\ensuremath{\mathcal{M}}}
\newcommand{\N}{\ensuremath{\mathcal{N}}}
\newcommand{\A}{\ensuremath{\mathcal{A}}}
\newcommand{\B}{\ensuremath{\mathcal{B}}}
\newcommand{\F}{\ensuremath{\mathcal{F}}}
\newcommand{\D}{\ensuremath{\mathcal{D}}}
\newcommand{\C}{\ensuremath{\mathcal{C}}}
\newcommand{\E}{\ensuremath{\mathcal{E}}}
\newcommand{\G}{\ensuremath{\mathcal{G}}}
\renewcommand{\P}{\mathbb{P}}
\renewcommand{\C}{\mathbb{C}}
\newcommand{\apspec}{AP-\ensuremath{Spec(\phi)} }
\newcommand{\jepspec}{JEP-\ensuremath{Spec(\phi)} }
\newcommand{\mmspec}{MM-\ensuremath{Spec(\phi)} }
\newcommand{\apspecpsi}{AP-\ensuremath{Spec(\psi)} }
\newcommand{\jepspecpsi}{JEP-\ensuremath{Spec(\psi)} }
\newcommand{\mmspecpsi}{MM-\ensuremath{Spec(\psi)} }
\newcommand{\specpsi}{\ensuremath{Spec(\psi)} }

\newcommand{\lang}[1]{\ensuremath{\mathcal{L}_{#1}}}
\newcommand{\langhat}{\ensuremath{\widehat{\lang{}}}}
\newcommand{\kappaplus}{\ensuremath{\kappa^{+}}}
\newcommand{\kappaplusplus}{\ensuremath{\kappa^{++}}}
\newcommand{\continuum}{\ensuremath{2^{\aleph_0}}}
\newcommand{\twoalephone}{\ensuremath{2^{\aleph_1}}}
\newcommand{\dom}{\mathop{\mathrm{dom}}\nolimits}
\newcommand{\ran}{\mathop{\mathrm{ran}}\nolimits}
\newcommand{\baleph}{\ensuremath{\B(\aleph_1)}}
\def\bk{\mbox{\boldmath $K$}}
\def\subm{\prec_{\mbox{\scriptsize \boldmath $K$}}}

\newtheorem{theorem}{Theorem}[section]
\newtheorem{lemma}[theorem]{Lemma}
\newtheorem{prop}[theorem]{Proposition}
\newtheorem{corollary}[theorem]{Corollary}
\theoremstyle{definition}
\newtheorem{definition}[theorem]{Definition}
\newtheorem{claim}[theorem]{Claim}
\newtheorem{subclaim}{Subclaim}
\newtheorem{note}{Note}
\newtheorem{observation}[theorem]{Observation}
\newtheorem{openq}[theorem]{Open Question}
\newtheorem{openqs}[theorem]{Open Questions}
\newtheorem{fact}[theorem]{Fact}
\newtheorem{example}[theorem]{Example}
\newtheorem{xca}[theorem]{Exercise}
\theoremstyle{remark}
\newtheorem{remark}[theorem]{Remark}
\newtheorem{remarks}[theorem]{Remarks}
\newtheorem{convention}[theorem]{Convention}
\numberwithin{equation}{section}

\title{Kurepa trees and spectra of $\mathcal{L}_{\omega_1,\omega}$-sentences}

\author{ Dima Sinapova}
\address{ University of Illinois at Chicago, Mathematics Department, 851 S Morgan St, Chicago, IL
60607}
\email{sinapova@math.uic.edu}
\thanks{Sinapova was partially supported by the National Science Foundation, grants DMS-1362485 and Career-1454945.}

\author {Ioannis  Souldatos}
\address{University of Detroit Mercy, Mathematics Department, 4001 McNichols Ave, Detroit, MI 48221}
\email{souldaio@udmercy.edu}



\date{\today}
\begin{abstract}
We use set-theoretic tools to make a model-theoretic contribution. In particular, we  construct a \emph{single} 
$\mathcal{L}_{\omega_1,\omega}$-sentence $\psi$ that codes Kurepa trees to prove the consistency of
the following:
\begin{enumerate}
 \item The spectrum of $\psi$ is consistently equal to $[\aleph_0,\aleph_{\omega_1}]$ and also consistently
equal to $[\aleph_0,2^{\aleph_1})$, where $2^{\aleph_1}$ is weakly inaccessible.

\item The amalgamation spectrum of $\psi$ is consistently equal to $[\aleph_1,\aleph_{\omega_1}]$ and
$[\aleph_1,2^{\aleph_1})$, where again $2^{\aleph_1}$ is weakly inaccessible.

This is the first example of an $\mathcal{L}_{\omega_1,\omega}$-sentence  whose spectrum and
amalgamation spectrum are consistently both right-open and right-closed. It also provides a positive answer to a question
in \cite{CharacterizableCardinals}.
 \item Consistently, $\psi$  has maximal models in finite, countable, and
uncountable many cardinalities. This complements the examples given in \cite{JEP} and \cite{Complete} of sentences with maximal
models in countably many cardinalities.

\item $2^{\aleph_0}<\aleph_{\omega_1}<2^{\aleph_1}$ and there exists an $\mathcal{L}_{\omega_1,\omega}$-sentence with
models in $\aleph_{\omega_1}$, but no models in $2^{\aleph_1}$.

This relates to a conjecture by Shelah that if $\aleph_{\omega_1}<2^{\aleph_0}$, then any $\mathcal{L}_{\omega_1,\omega}$-sentence
with a model of size $\aleph_{\omega_1}$ also has a model of size $2^{\aleph_0}$. Our result proves that $2^{\aleph_0}$ can not be
replaced by $2^{\aleph_1}$, even if $2^{\aleph_0}<\aleph_{\omega_1}$.

\end{enumerate}

\end{abstract}

\maketitle

\section{Introduction}
\begin{definition}\label{chardef} For an $\lomegaone$-sentence $\phi$, the (model-existence) \emph{spectrum} of $\phi$ is the set
\[Spec(\phi)=\{\kappa|\exists M\models\phi \text{ and } |M|=\kappa\}.\]
If $Spec(\phi)=[\aleph_0,\kappa]$, we say that $\phi$ \emph{characterizes} $\kappa$.

The \emph{amalgamation spectrum} of $\phi$, \apspec for short, is the set of all
cardinals $\kappa$ so that (i) $\phi$ has at least one model of size $\kappa$ and (ii) the models of $\phi$ of size 
$\kappa$ satisfy the amalgamation property\footnote{The amalgamation 
property is not unique, but rather it depens on the embeddings used. In this paper we consider amalgamation under the 
substructure relation. See definition of $(\bk, \subm)$ before Lemma \ref{obs1}}. Similarly define \jepspec
the \emph{joint embedding spectrum} of $\phi$.

The \emph{maximal models spectrum} of $\phi$ is the set
\[ \text{\mmspec=}\{\kappa|\exists M\models \phi \text{ and $M$ is maximal}\}.\]
\end{definition}

Morley and L\'{o}pez-Escobar independently established that all cardinals that are characterized by an $\lomegaone$- sentence are
smaller than $\bethomegaone$.\footnote{See \cite{KeislerModelTheory} for more details.}
\begin{theorem} [Morley, L\'{o}pez-Escobar] \label{HanfNumber}
Let $\Gamma$ be a countable set of sentences of $\lomegaone$.  If $\Gamma$ has models of
cardinality $\beth_\alpha$ for all
$\alpha<\omega_1$, then it has models of all infinite cardinalities.
\end{theorem}

 In 2002, Hjorth (\cite{HjorthsKnightsModel}) proved the following.

\begin{theorem}[Hjorth] \label{Thm:Hjorth} For all $\alpha<\omegaone$, $\alephalpha$ is
characterized by a complete
$\lomegaone$-sentence.
\end{theorem}

Combining Theorems \ref{HanfNumber} and \ref{Thm:Hjorth} we get under GCH that $\alephalpha$ is characterized by a complete
$\lomegaone$-sentence if and only if  $\alpha<\omegaone$.

Given these results, one might ask if it is consistent under the failure of GCH that $\alephomegaone$ is characterizable. The
answer is positive because one can simply force $\alephomegaone=\continuum$ and by \cite{MalitzsHanfNumber}, $\continuum$ can
be characterized by a complete $\lomegaone$-sentence. In order to block the possibility of characterizing $\alephomegaone$ by 
manipulating the value of e.g. $\continuum$,  in \cite{CharacterizableCardinals} the question was raised if any
cardinal outside the set $C$ which is defined to be the smallest set that contains $\aleph_0$ and is closed under 
successors, countable unions, countable products and powerset, can be characterized by an $\lomegaone$-sentence. Notice that $C$ 
contains all cardinals of the form $\aleph_\alpha$, for countable $\alpha$, and it follows by Theorems 
\ref{HanfNumber} and \ref{Thm:Hjorth}  that under GCH,  $\alephalpha\in C$ if and only if 
$\alpha$ is countable. Therefore, 
$\alephomegaone$ is the smallest cardinal that consistently does not belong to $C$. With this in mind, we consider the 
possibility of $\alephomegaone$ being a characterizable cardinal not in $C$. In Lemma \ref{smallestset} we prove that this can 
be consistently true, thus providing a positive 
answer to the question from \cite{CharacterizableCardinals}.

The importance of $\alephomegaone$  is also emphasized by a conjecture of Shelah that if 
$\alephomegaone<\continuum$, then any $\lomegaone$-sentence with a model of size $\alephomegaone$ also has a model of size 
$\continuum$. In \cite{ShelahsBorelSets}, Shelah proves the consistency of the conjecture. In Theorem \ref{spectrum} (see also 
Corollary \ref{shelahconj}) we prove that there is a sentence that consistently characterizes the maximum of $\continuum$ and 
$\alephomegaone$. It follows that 
consistently $\continuum<\alephomegaone<2^{\aleph_1}$ and there is an $\lomegaone$-sentence that characterizes $\alephomegaone$. 
This proves that $\continuum$ can not be replaced by $\twoalephone$ in the conjecture, even if
$2^{\aleph_0}<\alephomegaone$. Our example can not be used to refute Shelah's conjecture because our sentence always has a 
model of size $\continuum$. 

Another contribution of the current paper is in providing an $\lomegaone$-sentence whose spectrum (and amalgamation spectrum) 
exhibits quite peculiar behavior when we manipulate the underlying set theory. In particular, we provide the first example 
of an $\lomegaone$-sentence whose spectrum can consistently be right-open and right-closed. 
If $\kappa=\sup_{n\in\omega} \kappa_n$ 
 and $\phi_n$ characterizes $\kappa_n$, then $\bigvee_n \phi_n$ is a (necessarily incomplete) sentence that has spectrum 
$[\aleph_0,\kappa)$. All the previously known examples of $\lomegaone$-sentences with a right-open spectrum are variations of this 
example and have the form
$[\aleph_0,\kappa)$, with $\kappa$ of countable cofinality. In Corollary \ref{psistat} we prove that consistent with 
$2^{\aleph_1}$ being weakly inaccessible, our example has spectrum $Spec(\psi)=[\aleph_0,2^{\aleph_1})$. By Lemma 
\ref{bknotmax}, our methods can not be used to establish the consistency of
$Spec(\psi)=[\aleph_0,\aleph_{\omega_1})$, which remains open. It also open to find a \emph{complete} sentence with a right-open
spectrum, much less a complete sentence with a spectrum which is consistently right-open and right-closed.

In Theorem \ref{APSpectrum} we prove that the amalgamation spectrum contains exactly the uncountable cardinals in the 
model-existence spectrum. So, the same comments about the spectrum also apply to the amalgamation spectrum. 
Moreover, by manipulating the size of $\twoalephone$, Corollary \ref{Cor:APSpectra} implies that for
$\kappa$ regular with $\aleph_2\le \kappa\le \twoalephone$, $\kappa$-amalgamation for $\lomegaone$-sentences is non-absolute for
models of ZFC. This complements the result in \cite{MilovichSModelExistence} that under GCH, $\alephalpha$-amalgamation (under 
substructure again) is not absolute for $1<\alpha<\omega$. It is interesting to notice that the examples presented both in this 
paper and in \cite{MilovichSModelExistence} can not settle  the absoluteness question for $\aleph_1$-amalgamation, which 
remains open. We also notice that the amalgamation spectra in this paper are always non-empty, as they 
always contain $\aleph_1$. This is in contrast with the examples in \cite{MilovichSModelExistence}, which maybe consistently 
empty or non-empty.

Investigating the maximal models spectrum of our example, \mmspecpsi,  we prove in Theorem \ref{spectrum} that 
\mmspecpsi is the set of all cardinals $\kappa$ such that $\kappa$ is either equal to $\aleph_1$, or equal to $\continuum$,
or there exists a Kurepa tree with exactly $\kappa$-many branches. Manipulating the cardinalities on which there exist Kurepa
trees, we can produce maximal models in finite, countable, and uncountable many cardinalities, plus we can carefully control 
which these cardinalities are. Results concerning the maximal models spectrum can also be found in \cite{JEP} and 
\cite{Complete}. From \cite{JEP} we know that  if $\overline{\lambda}=\langle \lambda_i: i\le \alpha <\aleph_1\rangle$ is a 
strictly increasing
 sequence of characterizable cardinals, there is an $\lomegaone$-sentence $\phi_{\overline{\lambda}}$, which has the maximum 
number (that is  $2^{\lambda_i^+}$) non-isomorphic maximal models in cardinalities $\lambda_i^+$, for all $i\le\alpha$, and no
maximal models in any other cardinality. The sentence $\phi_{\overline{\lambda}}$ is incomplete. In contrast, the examples 
produced in \cite{Complete} are complete sentences. The example in the current paper is also an incomplete sentence. 

In \cite{GrossbergClassicationAECs}, Grossberg conjectured that for every Abstract Elementary Class (AEC) $\bk$ there exists a 
cardinal $\mu(LS(\bk))$ such that if $\bk$ has the $\mu(LS(\bk))$-amalgamation property then $\bk$ has the $\lambda$-amalgamation 
property 
for all $\lambda\ge \mu(LS(\bk))$. This cardinal $\mu(LS(\bk))$ is called the Hanf number for amalgamation. Baldwin and Boney in 
\cite{BaldwinBoneyHanfNumbers} proved the existence of a Hanf number for amalgamation, but their definition of a Hanf 
number is different than the one that Grossberg conjectured. Their result states that if $\mu$ is a 
strongly compact cardinal, $\bk$ is an AEC with $LS(\bk)<\mu$, and $\bk$ satisfies amalgamation cofinally below $\mu$, then $\bk$ 
satisfies amalgamation in all cardinals $\ge \mu$. The question was raised if the strongly compact upper bound was optimal for 
the Hanf number for amalgamation. 

Far from being optimal, a lower bound for the Hanf number for amalgamation is given by the examples in \cite{KLH}. There
it is proved that for every cardinal $\kappa$ and every $\alpha$ with $\kappa\le\alpha<\kappaplus$ there exists an AEC 
$W_\alpha$ with countable $LS(W_\alpha)$, the vocabulary is of size $\kappa$, $W_\alpha$ satisfies amalgamation up 
to $\beth_{\alpha}$, but fails amalgamation in $\beth_{\kappaplus}$.

Although in this paper we do not claim to provide any lower bounds for the Hanf number for amalgamation, we believe that 
the idea of coding $\kappa$-Kurepa trees, for $\kappa>\aleph_1$, can be pursued further with the goal of 
producing amalgamation spectra that exhibit similar behavior to the ones in this paper, but in higher cardinalities.

Finally, we want to emphasize again the fact that in this paper we produce a \emph{single} $\lomegaone$-sentence $\psi$ and all 
the results mentioned above follow by considering the various spectra of this sentence $\psi$ and how set-theory affects these 
spectra.  In particular, we prove the consistency of the following.

\begin{enumerate}
 \item $\continuum<\aleph_{\omega_1}$ and $\specpsi=[\aleph_0,\alephomegaone]$.
\item $\continuum<2^{\aleph_1}$, $2^{\aleph_1}$ is weakly inaccessible and $Spec(\psi)=[\aleph_0,2^{\aleph_1})$.

\item $\continuum<\aleph_{\omega_1}$ and \apspecpsi=$[\aleph_1,\alephomegaone]$.
\item $\continuum<2^{\aleph_1}$, $2^{\aleph_1}$ is weakly inaccessible and \apspecpsi=$[\aleph_1,2^{\aleph_1})$.


\item $\continuum<\aleph_{\omega_1}<2^{\aleph_1}$ and there exists an $\lomegaone$-sentence with
models in $\aleph_{\omega_1}$, but no models in $\twoalephone$.
\end{enumerate}

Section \ref{modeltheory} contains the description of the sentence $\psi$ and the results about the model-theoretic properties of
$\psi$.
Section \ref{Consistency} contains the consistency results.

\section{Kurepa Trees and $\lomegaone$}\label{modeltheory}

The reader can consult \cite{JechsSetTheory} about trees. The following definition summarizes all
that we
will use in this paper.
\begin{definition} Assume $\kappa$ is an infinite cardinal. A
$\kappa$-tree has height $\kappa$ and  each level has at most $<\kappa$ elements.
A $\kappa$-Kurepa tree is a $\kappa$-tree with at least $\kappaplus$ branches of
height $\kappa$. Kurepa trees, with no $\kappa$ specified, refer to $\aleph_1$-Kurepa trees.
For this paper we will assume that $\kappa$-Kurepa trees are \emph{pruned},
i.e. all maximal branches have height $\kappa^+$.

If $\lambda\ge\kappaplus$, a $(\kappa,\lambda)$-Kurepa tree is a $\kappa$-Kurepa tree with
exactly $\lambda$ branches of height $\kappa$.
$KH(\kappa,\lambda)$ is the statement that there exists a $(\kappa,\lambda)$-Kurepa tree.

Define $\B(\kappa)=\sup\{\lambda|KH(\kappa,\lambda) \text{ holds}\}$ and $\B=\B(\aleph_1)$.
\end{definition}

If $\kappa$-Kurepa trees exist, it is immediate that $\B(\kappa)$
is always between $\kappaplus$ and $2^\kappa$. We are interested in the case where $\B$ is a supremum but not a maximum. The next
lemma proves some restrictions when $\B$ is not a maximum. In Section \ref{Consistency} we prove that it is consistent with
ZFC that $\B$ is not a maximum.

\begin{lemma}\label{bknotmax}
If $\B(\kappa)$ is not a maximum, then
$cf(\B(\kappa))\ge\kappaplus$.
\end{lemma}
\begin{proof}
Towards contradiction assume that $cf(\B(\kappa))=\mu\le\kappa$. Let $\B(\kappa)=sup_{i\in\mu} \kappa_i$ and let
$(T_i)_{i\in\mu}$ be a collection of $\kappa$-Kurepa trees, where each $T_i$ has exactly $\kappa_i$-many cofinal branches. Create
 new
$\kappa$-Kurepa trees $S_i$  by induction on $i\le \mu$: $S_0$ equals $T_0$; $S_{i+1}$ equals the disjoint union of $S_i$
together with a
copy of $T_{i+1}$, arranged so that the $j^{th}$ level $T_{i+1}$ coincides with the $(j+i)^{th}$ level of $S_i$. At limit
stages take unions. We leave the verification to the reader that $S_{\mu}$ is
a $\kappa$-Kurepa tree with exactly $sup_{i\in\mu} \kappa_i=\kappa$ cofinal branches, contradicting the fact that $\B(\kappa)$ is
not a maximum.
\end{proof}

\begin{definition} Let $\kappa\le\lambda$ be infinite cardinals.
A sentence $\sigma$ in a language with a unary predicate $P$ admits $(\lambda,\kappa)$, if $\sigma$
has a model $M$ such that $|M|=\lambda$ and $|P^M|=\kappa$. In this case, we will say that $M$ is
of type $(\lambda,\kappa)$.
\end{definition}

From \cite{ChangKeislerModelTheory}, theorem $7.2.13$, we know
\begin{theorem}\label{Thm:ChangKeisler}There is a (first-order) sentence $\sigma$ such that for all
infinite cardinals
$\kappa$, $\sigma$ admits $(\kappaplusplus,\kappa)$ iff $KH(\kappaplus,\kappaplusplus)$.
\end{theorem}

The idea of the proof of Theorem \ref{Thm:ChangKeisler} is also given  in \cite{ChangKeislerModelTheory}, but the details are 
left to the reader. Since the amalgamation property is heavily dependent on the particular formulation\footnote{This means it 
is possible to come up with two different sentences $\sigma_1$, $\sigma_2$ that both satisfy Theorem \ref{Thm:ChangKeisler}, but 
which exhibit different amalgamation spectra.}, we give a full description of sentence $\sigma$, including the vocabulary, before 
we proceed.

The vocabulary $\tau$ consists of the constant $0$, the unary symbols $P,L$, the binary symbols $S,V,T,<,\prec,H$, 
and the ternary symbols $F,G$. The idea is to build a $\kappa$-tree. $L$ is a set that corresponds to the ``levels'' of the tree. 
$L$
is linearly ordered by $<$ and $0$ is its minimum element. $L$ may or may not have a maximum element. Every element $a\in L$
that is not the maximum element (if any) has a successor $b$ that satisfies $S(a,b)$. We will 
freely denote the successor of $a$ by $S(a)$. The maximum element is not a successor. For every
$a\in L$, $V(a,\cdot)$ is the set of nodes at level $a$ and we assume that $V(a,\cdot)$ is disjoint from $L$. If $V(a,x)$, we will 
say that $x$ is at level $a$ and we may write $x\in V(a)$.  

$T$ is a tree ordering on $V=\bigcup_{a\in L} V(a)$. If $T(x,y)$, then $x$ is at some level strictly less than the
level of $y$. If $y\in V(a)$ and $b<a$, there is some $x$ so that $x\in V(b)$ and $T(x,y)$. If $a$ is a limit, that is 
neither a successor nor $0$, then two distinct elements in $V(a)$ can not have the same predecessors. If $m$ is the maximum 
element of $L$, $V(m)$ is the set of maximal branches through the 
tree. Both ``the height of $T$'' and ``the height of $L$'' refer to the order type of $(L,<)$. Although it is not
necessary for Theorem \ref{Thm:ChangKeisler}, we can stipulate that the Kurepa tree is pruned.

There is an old trick to bound the size of a linear order by $\kappaplus$ by bounding the size of each initial segment by 
$\kappa$. We use this trick to bound the size of $L$ by $\kappaplus$ and the size of $V(m)$ by $\kappaplusplus$. 

For every $a\in L$ that is not the maximum element of $L$ (if any),  there is a surjection
$F(a,\cdot,\cdot)$ from the predicate $P$ to $L_{\le a}=\{b\in L|b\le a\}$ and another surjection $G(a,\cdot,\cdot)$ from $P$ to
$V(a)$. This bounds the size of every initial segment $L_{\le a}$ and the size every $V(a)$ by $|P|$, and in return the size 
of $L$ by $|P|^+$.

If $L$ has a maximum element $m$, we linearly order the set of maximal branches $V(m)$ by $\prec$ so that there 
is no maximum element and require that  $H$ is a surjection from $L$ to each initial segment. It follows that the 
size of $V(m)$ is bounded by $|L|^+$.

The disjoint union of $L,P,\cup_{a\in L} V(a)$ gives the domain of the model.

Call $\sigma$ the (first-order) sentence that stipulates all the above.
If $\M$ is a model of $\sigma$ with $P^\M$ of size $\kappa$ and $\M$ has size $\kappaplusplus$,
then $L$ has a maximum element $m$, $|V^\M(m)|=\kappaplusplus$ and $|L^\M|=\kappaplus$. So, $\M$ codes a $|P|$-Kurepa tree. 
Conversely, from every $\kappa$-Kurepa tree, we can define a model of $\sigma$ with $|P|=\kappa$. This 
proves Theorem \ref{Thm:ChangKeisler}.

We want to emphasize here the fact that since well-orderings can not be characterized by an
$\lomegaone$-sentence, it is unavoidable that we will be working with non-well-ordered trees.
However, using an $\lomegaone$-sentence we can express the fact that $P$ is countably infinite. There are many ways to
do this. The simplest way is to introduce countably many new constants $(c_n)_{n\in\omega}$ and require that $\forall x, \;
P(x)\leftrightarrow \bigvee_n x=c_n$. 
Let $\phi$ be the conjunction of $\sigma$ together with the requirement that $P$ is countably
infinite. Then $\phi$ has models of size $\aleph_2$ iff there exist a Kurepa tree of size
$\aleph_2$ iff $KH(\aleph_1,\aleph_2)$.

Fix some $n\ge 2$. Then the above construction of $\sigma$ (and $\phi$) can be modified to
produce a first-order sentence $\sigma_n$ and the corresponding $\lomegaone$-sentence $\phi_n$ so
that $\phi_n$ has a model of size $\aleph_n$ iff there exist a Kurepa tree of size
$\aleph_n$ iff $KH(\aleph_1,\aleph_n)$. The argument breaks down at $\aleph_\omega$. Since
we will be dealing with Kurepa trees of size potentially larger than $\aleph_{\omega}$, we must make some modifications.

Let $\tau'$ be equal to $\tau$ with the symbols $\prec, H$ removed. Let $\sigma'$ be equal to
$\sigma$ with all requirements that refer to $\prec, H$ removed. Let $\psi$ be the conjunction of
$\sigma'$ and the requirement that $P$ is countably infinite as formulated above\footnote{Our formulation has the property that 
that if $M\subset N$ are two models of $\psi$, then $P^M=P^N$. That is, $P$ does not grow.}. 
Notice that both $\phi$ and $\psi$ are not complete sentences.

For any $\lambda\ge\aleph_2$, any
$(\aleph_1,\lambda)$-Kurepa tree gives rise to a model of $\psi$, but unfortunately, there are
models of $\psi$ of size $2^{\aleph_0}$, that do not code a Kurepa tree. These trees have countable height.  For instance 
consider the tree $(\omega^{\le\omega},\subset)$
which contains $2^{\aleph_0}$ many maximal branches.

The dividing line for models of $\psi$ to code Kurepa trees is the size of $L$. By definition
$L$
is $\aleph_1$-like, i.e. every initial segment has countable size. If in addition $L$ is
uncountable, then we can embed $\omega_1$ cofinally into $L$.\footnote{The embedding is not
necessarily continuous, i.e. it may not respect limits.} Hence, every model of $\psi$
of size $\ge\aleph_2$ and for which $L$ is uncountable, codes a Kurepa tree. Otherwise, the model
does not code a Kurepa tree.

Let $\bk$ be the collection of all models of $\psi$, equipped with the substructure relation. I.e.
for $M,N\in \bk$, $M\subm N$ if  $M\subset N$.

 \begin{lemma}\label{obs1}
If $M\subm N$, then
\begin{enumerate}
 \item $L^M$ is an initial segment of $L^N$;
 \item for every non-maximal $a\in L^M$, $V^M(a,\cdot)$ equals $V^N(a,\cdot)$. 
\item  the tree ordering is preserved.
\end{enumerate}
We will express $(1)-(3)$ by saying that ``(the tree defined by ) $M$ is an \emph{initial segment}
of (the
tree defined by) $N$''.
\begin{proof}

For part (1), first recall that if $M\subm N$, then $P^M=P^N$. Towards contradiction, assume 
that $L^N$ contains some point $x\notin L^M$ and there exists some $y\in L^M$ such that $x<y$. Then the function
$F^M(y,\cdot,\cdot)$ defined  on $M$ disagrees with the function $F^N(y,\cdot,\cdot)$ defined  on
$N$. Contradiction.

For part (2), the argument is similar to the argument for part (1), using the functions $G^M(y,\cdot,\cdot)$ and
$G^N(y,\cdot,\cdot)$ this time.

Part (3) is immediate from the definition.
\end{proof}

\end{lemma}

\begin{corollary}\label{initseg} Assume $M$ is an initial segment of $N$. Then:
\begin{itemize}
 \item If $L^{N}=L^{M}$, then $N$ differs from $M$ only in the maximal branches it
contains.
\item If $L^{M}$ is uncountable and $L^{N}$ is a strict end-extension of $L^{M}$, then $L^{M}$ does not have a maximum element 
and $L^{N}$ is a one-point end extension of $L^M$.
\item If $L^{M}$ has a maximum element and $L^{N}$ is a strict end-extension of $L^{M}$, then $M$ must be countable.
 \end{itemize}
\end{corollary}
\begin{proof} Part (1) is immediate from Lemma \ref{obs1} (2). 

For part (2), observe that if $L^M$ had a maximum element, then $L^N$ would not be $\aleph_1$-like. A similar contradiction we 
derive if we assume that $L^N$ contains at least two new points not in $L^M$. 

For part (3) first observe that $L^M$ is countable. Otherwise we get a contradiction from part (2). Then the 
statement follows 
from the requirement that in all models of $\psi$ all non-maximal levels are countable. 
\end{proof}

\begin{lemma} $(\bk,\subm)$ is an AEC with countable Lowenheim-Skolem number.
\end{lemma}
\begin{proof} We verify closure under unions of chains. The other properties for being an AEC are immediate. The requirements set 
forth by $\sigma'$ can be written using a $\forall\exists$ sentence. Therefore $\sigma'$ is preserved under unions of chains. The 
requirement that $P$ is countable was formulated in such a way that does not allow $P$ to gain new elements. Therefore, this 
requirement is also preserved under unions of chains.

To prove that $LS(\bk)$ is countable, consider $N$ a model of $\psi$ and some countable $X\subset N$. We need some 
countable model $M$ of $\psi$ that contains $X$ and $M\subm N$. If $L^N$ is countable, then take $M\subset N$ so that the 
domain of $M$ is the union of $P^N$, $L^N$, every non-maximal level of $V^N$, and all maximal branches from $V^N$ that 
belong to $X$ (if any). If $L^N$ is uncountable, then without loss of generality assume that it has a maximum element $m$. Let 
$L_0$ be a strict initial segment of $L^N$ that does not have a maximum element and such that if $x$ is a non-maximal element in 
$X$, then $x$ belongs to a level in $L_0$, and if $x,y$ are two maximal elements in $X\cap V^N$, then $x,y$ differ at some 
level in $L_0$. Take $M\subset N$ so that the domain of $M$ is the union of $P^N$, $L_0\cup \{m\}$, the elements 
of $V^N$ that belong to a level in $L_0$, and all maximal branches from $V^N$ that belong to $X$(if any). The reader can verify 
that in either case $M$ is a model in $\bk$ and contains $X$. 
\end{proof}

\begin{convention}
For the rest of the paper, when we talk about ``the models of $\psi$'' we will mean
$(\bk,\subm)$.
\end{convention}

The next theorem characterizes the spectrum and the
maximal models spectrum of $\psi$.

\begin{theorem}\label{spectrum}~ The spectrum of $\psi$ is characterized by the following
properties:
\begin{enumerate}
 \item $[\aleph_0,2^{\aleph_0}]\subset Spec(\psi)$;
 \item if there exists a Kurepa tree with $\kappa$ many branches, then
$[\aleph_0,\kappa]\subset Spec(\psi)$;
 \item no cardinal belongs to $Spec(\psi)$ except
those required by (1) and (2). I.e. if $\psi$ has a model of size $\kappa$, then either
$\kappa\le\continuum$, or there exists a Kurepa tree which $\kappa$ many branches.
\end{enumerate}

The maximal models spectrum of $\psi$ is characterized by the following:
\begin{enumerate}
\setcounter{enumi}{3}
 \item $\psi$ has maximal models in cardinalities $2^{\aleph_0}$ and $\aleph_1$;
  \item if there exists a Kurepa tree with exactly $\kappa$ many branches, then $\psi$ has a
maximal model in $\kappa$;
\item $\psi$ has maximal models only in those cardinalities required by (4) and (5).
\end{enumerate}
\end{theorem}

\begin{proof} For (1) and (2), we observed already that $(\omega^{\le\omega},\subset)$ is a model
of $\psi$ and that every Kurepa tree gives rise to a model of $\psi$.

For (3), let $N$ be a model of $\psi$ of size $\kappa$. If $L^N$ is countable, then $N$ has size
$\le\continuum$. If $L^N$ is uncountable, then $N$ codes a $\kappa$-tree with
$|N|$-many branches. By the proof of theorem \ref{Thm:ChangKeisler}, this is a Kurepa tree, assuming that $|N|\ge \aleph_2$.
Otherwise, $|N|\le\aleph_1\le\continuum$.

To prove (4), notice that $(\omega^{\le\omega},\subset)$ is a maximal model of
$\psi$, and it is easy to construct trees with height $\omega_1$ and
$\le\aleph_1$ many branches.

Now, let $N$ code a Kurepa tree with exactly $\kappa$ many
branches. By Corollary \ref{initseg}, $N$ is maximal. This proves (5).

For (6), let $N$ be a maximal model. If $N$ was countable, we could end-extend it and
it would not be maximal. So, $N$ must be uncountable. We split into two cases depending on
the size of $L^N$. If $L^N$ is countable, assume without loss of generality, that it has
height $\omega$ (otherwise consider a cofinal subset of order type $\omega$). So, $(V^N,T^N)$ is a
pruned tree which is a subtree of $\omega^{\le\omega}$. The set of maximal branches through
$(V^N,T^N)$ is a closed subset of the Baire space (cf. \cite{KechrisClassical}, Proposition 2.4).
Since closed subsets of $\omega^\omega$ have size either $\aleph_0$ or $\continuum$, we conclude
that $N$ has size $\continuum$. The second case is when $L^N$ is uncountable. If $N$ has size $\aleph_1$, we are done. If $|N|\ge
\aleph_2$, by Corollary \ref{initseg} and maximality of $N$, the tree defined by $N$ contains exactly $|N|$ many maximal
branches. Therefore, $N$ defines a Kurepa tree.
\end{proof}

Recall that $\B=\baleph$ is the supremum of the size of Kurepa trees.

\begin{corollary}\label{spectrumcor}
\begin{enumerate}
 \item If there are no Kurepa trees, then $\specpsi$ equals $[\aleph_0,2^{\aleph_0}]$ and \mmspecpsi equals
$\{\aleph_1,\continuum\}$.
 \item If $\B$ is a maximum, i.e. there is a Kurepa tree of size
$\B$, then $\psi$ characterizes $\max\{2^{\aleph_0},\B\}$.
\item If $\B$ is not a maximum, then $Spec(\psi)$ equals either $[\aleph_0,2^{\aleph_0}]$ or $[\aleph_0,\B)$, whichever is
greater. Moreover $\psi$ has maximal models in $\aleph_1$, $2^{\aleph_0}$ and in cofinally many cardinalities
below $\B$.
\end{enumerate}
\end{corollary}
\begin{proof} (1) and (2) follow immediately from Theorem \ref{spectrum}. We only establish (3).
If $\B$ is not a maximum, then $[\aleph_0,\B)\subset Spec(\psi)$ and $Spec(\psi)$
equals either $[\aleph_0,2^{\aleph_0}]$ or $[\aleph_0,\B)$, whichever is greater. For the the last
assertion, assume $\B=\sup_i \kappa_i$ and for each $i$, there is a Kurepa tree with $\kappa_i$ many branches. Then
each $\kappa_i$ is in the \mmspecpsi by Theorem \ref{spectrum} (5).
\end{proof}

In Section \ref{Consistency} we prove the following consistency results.
\begin{theorem}\label{constthrm}
The following are consistent with ZFC:
\begin{enumerate}[(i)]
 \item ($2^{\aleph_0}<\aleph_{\omega_1}=\B<2^{\aleph_1}$) +``$\B$ is a maximum``, i.e. there
exists a Kurepa
tree of size $\alephs{\omega_1}$''.

\item ($\aleph_{\omega_1}=\B<2^{\aleph_0}$)+ ``$\B$ is a maximum''.
  \item ($2^{\aleph_0}<\B=2^{\aleph_1}$) + ``$2^{\aleph_1}$ is weakly inaccessible
   + ``for every $\kappa<\twoalephone$  there is a Kurepa tree with (at least) $\kappa$-many maximal branches, but no Kurepa tree
has exactly $\twoalephone$-many branches.''
 \end{enumerate}
 \end{theorem}

Moreover, in (i) and (ii) we can replace $\alephomegaone$ by most cardinals below or
equal to $2^{\aleph_1}$ and $\continuum$ respectively. We also note that in (iii) we use a Mahlo cardinal to prove the 
consistency result.

From \cite{JinKurepaThick} we know the consistency of the following:
\begin{theorem}[R. Jin]\label{thickkurepa}
Assume the existence of two strongly inaccessible cardinals. It is consistent with $CH$ (or
$\neg CH$ ) plus $2^{\aleph_1}>\aleph_2$ that there exists a
Kurepa tree with $2^{\aleph_1}$ many branches and no $\omega_1$-trees have $\lambda$-many branches
for some $\lambda$ strictly between $\aleph_1$ and $2^{\aleph_1}$. In particular, no Kurepa trees
have less than $2^{\aleph_1}$ many branches.
\end{theorem}
It follows from the proof of Theorem \ref{thickkurepa} that if $\lambda$ is regular cardinal
above $\aleph_2$, we can force the size of $2^{\aleph_1}$ to equal $\lambda$.

\begin{corollary}\label{psistat}
 There exists an \lomegaone-sentence $\psi$ such that it is consistent with ZFC that
\begin{enumerate}
 \item  $\psi$ characterizes $2^{\aleph_0}$;
 \item CH (or $\neg CH$ ), $2^{\aleph_1}$ is a regular cardinal greater than $\aleph_2$ and $\psi$ characterizes
$2^{\aleph_1}$;
 \item $2^{\aleph_0}<\aleph_{\omega_1}$ and $\psi$ characterizes $\aleph_{\omega_1}$; and
  \item $2^{\aleph_0}<2^{\aleph_1}$, $2^{\aleph_1}$ is weakly inaccessible and $Spec(\psi)=[\aleph_0,2^{\aleph_1})$.
 \end{enumerate}
 For the same $\psi$ it is consistent with ZFC that
 \begin{enumerate}
 \setcounter{enumi}{4}
  \item \mmspecpsi=$\{\aleph_1,\continuum\}$
    \item \mmspecpsi=$\{\aleph_1,\continuum,2^{\aleph_1}\}$
    \item $2^{\aleph_0}<2^{\aleph_1}$, $2^{\aleph_1}$ is weakly inaccessible and
\mmspecpsi is a cofinal subset of $[\aleph_1,2^{\aleph_1})$.
 \end{enumerate}
\end{corollary}
\begin{proof}
 The result follows from Theorem \ref{spectrum} and Corollary \ref{spectrumcor} using the appropriate model of ZFC each time. For
(1) consider a model with no Kurepa trees, or a model where $\B<\continuum$, e.g. case (ii) of Theorem $\ref{constthrm}$.
For (2) and (6) use Theorem \ref{thickkurepa}. For (3),(4), use Theorem \ref{constthrm} cases (i),(iii) respectively. For (5)
consider a model of ZFC with no Kurepa trees.  For (7) use Theorem \ref{constthrm}
case (iii) again.
\end{proof}

\begin{corollary}\label{shelahconj} It is consistent with ZFC that
$\continuum<\aleph_{\omega_1}<2^{\aleph_1}$ and there exists an $\lomegaone$-sentence with
models in $\aleph_{\omega_1}$, but no models in $\twoalephone$.
\end{corollary}

\vspace{6pt}
\subsection{Amalgamation and Joint Embedding Spectra}~

In this section we provide the amalgamation and joint embedding spectrum of models of $\psi$.

The following characterizes JEP-$Spec(\psi)$ and AP-$Spec(\psi)$.
\begin{theorem} \label{APSpectrum}
\begin{enumerate}
 \item $(\bk,\subm)$ fails JEP in all cardinals;
 \item $(\bk,\subm)$ satisfies AP for all uncountable cardinals that belong to $\specpsi$, but fails AP in
$\aleph_0$.
\end{enumerate}
\end{theorem}
\begin{proof} The first observation is that in all cardinalities there exists two linear orders
$L^M$, $L^N$ none of which is an initial segment of the other. By Lemma \ref{obs1}(1), $M,N$
can not be be jointly embedded to some larger structure in $\bk$. So, JEP fails in all
cardinals.

A similar argument to JEP proves that there exist three countable linear orders
$L^{M_0}$, $L^{M_1}$, $L^{M_2}$ such that $L^{M_0}$ is an initial segment of both $L^{M_1}$ and
$L^{M_2}$, and the triple $(L^{M_0},L^{M_1},L^{M_2})$ can not be amalgamated. This proves that
amalgamation fails in $\aleph_0$.

Now, assume that $M,N$ are uncountable models of $\psi$ and $M\subm N$. By Corollary
\ref{initseg}, either $L^M=L^N$ or $L^N$ is a one-point extension of $L^N$. In either case $M,N$ agree on all 
non-maximal levels, but $N$ may contain more maximal branches. We use this observation to prove amalgamation.

Let $M_0,M_1,M_2$ be uncountable models in $\bk$ with $M_0\subm M_1,M_2$. Define the amalgam $N$ of $(M_0,M_1,M_2)$ to be the 
union of $M_0$ together with all maximal branches in $M_1$
and $M_2$ (if any). If two maximal branches have exactly the same predecessors, we identify them. It follows
that $N$ is a structure in $\bk$ and $M_i\subm N$, $i=1,2,3$.
\end{proof}
Notice that, in general, the amalgamation is not disjoint, since both $M_1$ and $M_2$ may contain
the same maximal branch.

\begin{corollary}\label{Cor:APSpectra}
The following are consistent:
\begin{enumerate}
 \item AP-$Spec(\psi)=[\aleph_1,\continuum]$;
 \item AP-$Spec(\psi)=[\aleph_1,\twoalephone]$;
 \item $\continuum<\alephomegaone$ and AP-$Spec(\psi)=[\aleph_1,\aleph_{\omega_1}]$; and
 \item $2^{\aleph_0}<2^{\aleph_1}$, $2^{\aleph_1}$ is weakly inaccessible and   AP-$Spec(\psi)=[\aleph_1,2^{\aleph_1})$.
\end{enumerate}
\end{corollary}
\begin{proof}
 The result follows from Theorem \ref{APSpectrum} and Corollary \ref{psistat}.
\end{proof}

It follows from Corollary \ref{Cor:APSpectra} that if $\aleph_2\le\kappa\le\twoalephone$ is a regular cardinal, then the
$\kappa$-amalgamation property for $\lomegaone$-sentences is not absolute for models of ZFC. The result is useful especially
under the failure of  GCH, since GCH implies that $\twoalephone=\aleph_2$. 

In addition, it is an easy application of Shoenfield's absoluteness that $\aleph_0$-amalgamation is an absolute property
for models of ZFC. The question for $\aleph_1$-amalgamation remains open. Recall that by
\cite{FriedmanEtAllNonAbsolutenessOfModelExistence}, model-existence in $\aleph_1$ for $\lomegaone$-sentences is
an absolute property for models of ZFC and by \cite{MilovichSModelExistence}, if we assume GCH, $\aleph_\alpha$-amalgamation is 
non-absolute for $2\le\alpha<\omega_1$. 
\begin{openq}~
Is $\aleph_1$-amalgamation  for $\lomegaone$-sentences absolute for models of ZFC?
\end{openq}

\section{Consistency Results}\label{Consistency}

In this section we prove the consistency results announced by Theorem \ref{constthrm}. Recall that $\B$ is the supremum of
$\{\lambda|\text{there exists a Kurepa tree with $\lambda$-many branches}\}$.

\begin{theorem}\label{constpart1} It is consistent with ZFC that:
\begin{enumerate}
\item $2^{\aleph_0}<\aleph_{\omega_1}=\B<2^{\aleph_1}$ and there  exists a Kurepa tree of size $\alephs{\omega_1}$.
\item $\aleph_{\omega_1}=\B<2^{\aleph_0}$ and there  exists a Kurepa tree of size $\alephs{\omega_1}$.
\end{enumerate}
\end{theorem}

We start with a model $V_0$ of ZFC+GCH.  Let $\P$ be the standard $\sigma$-closed poset for adding a Kurepa tree $K$
with $\aleph_{\omega_1}$-many $\omega_1$-branches. More precisely, conditions in $\P$ are of the form $(t,f)$, where:
\begin{itemize}
\item
$t$ is a tree of height $\beta+1$ for some $\beta<\omega_1$ and countable levels;
\item
$f$ is a function with $\dom(f)\subset \aleph_{\omega_1}$, $|\dom(f)|=\omega$, and $\ran(f)= t_\beta$, where $t_\beta$ is the
$\beta$-th level of $t$.
\end{itemize}
Intuitively, $t$ is an initial segment of the generically added tree, and each $f(\delta)$ determines where the $\delta$-th branch
 intersects the tree at level $\beta$.
The order is defined as follows: $(u,g)\leq (t,f)$ if,
\begin{itemize}
\item
$t$ is an initial segment of $u$,
$\dom(f)\subset\dom(g)$,
\item
for every $\delta\in\dom(f)$, either $f(\delta)=g(\delta)$ (if $t$ and $u$ have the same height) or $f(\delta)<_u g(\delta)$.
\end{itemize}
We have that $\P$ is countably closed and has the $\aleph_2$-chain condition.
Suppose that $H$ is $\P$-generic over $V_0$. Then  $\bigcup_{(t,f)\in H}t$ is a Kurepa tree with $\aleph_{\omega_1}$-many
branches,  where for $\delta<\aleph_{\omega_1}$, the $\delta$-th branch is given by $\bigcup_{(t,f)\in H,\delta\in\dom(f)}
f(\delta)$. These branches are distinct by standards density arguments. Also, note that since $|B|$ cannot exceed $2^{\aleph_1}$
and in this model $\aleph_{\omega_1}=|B|$, we have that $\aleph_{\omega_1}<2^{\aleph_1}$.

The model $V_0[H]$ proves part (1) of
Theorem \ref{constpart1}.
The same model also answers positively a question raised in
\cite{CharacterizableCardinals}. The question was whether any
cardinal outside the smallest set which contains $\aleph_0$ and which is closed under successors,
countable unions, countable products and powerset, can be characterized by an $\lomegaone$-sentence. $\alephomegaone$ is
consistently such an example.

\begin{lemma}\label{smallestset} Let $C$ be the smallest set of cardinals that contains $\aleph_0$ and is closed under
successors, countable unions, countable products and powerset. In $V_0[H]$, the set $C$ does not
contain $\aleph_{\omega_1}$.
\end{lemma}
\begin{proof}
Since $2^{\omega}<\aleph_{\omega_1}<2^{\omega_1}$, it is enough to show that $\aleph_{\omega_1}$ is
not the countable product of countable cardinals. Suppose $\langle\alpha_n\mid n<\omega\rangle$ is
an increasing sequence of countable ordinals and let $\alpha=\sup_n\alpha_n+1$. Then
$\prod_n\aleph_{\alpha_n}=(\aleph_\alpha)^\omega=\aleph_{\alpha+1}$.
\end{proof}

Let $\C=Add(\omega,\aleph_{\omega_1+1})$ denote the standard poset for adding   $\aleph_{\omega_1+1}$-many Cohen reals. Suppose
$G$ is $\C$-generic over $V:=V_0[H]$. (Note that $\C$ is interpreted the same in $V_0$ and in $V$ and  that actually genericity
over $V_0$ implies genericity over $V$ by the ccc.)

We claim that the forcing extension $V_0[H\times G]=V[G]$ satisfies $\continuum>\alephomegaone=\B$.
Let $T$ be a Kurepa tree in  $V[G]$. Denote $\C_{\omega_1}:= Add(\omega, \omega_1)$, i.e. the Cohen poset
for adding $\omega_1$ many reals. The following fact is standard and can be found in \cite{JechsSetTheory}, but we give the  proof
for completeness.
\begin{lemma}\label{submodel}
There is a generic $\bar{G}$ for $\C_{\omega_1}$, such that $V[\bar{G}]\subset V[G]$ and $T\in V[\bar{G}]$.

\end{lemma}
\begin{proof}
Since $T$ is a tree of height $\omega_1$ and countable levels, we can index the nodes of $T$ by $\langle\alpha, n\rangle$, for
$\alpha<\omega_1$ and $n<\omega$, where the first coordinate denotes the level of the node. In particular each level
$T_\alpha=\{\alpha\}\times\omega$. Note that this is in the ground  model (although of course the relation $<_T$ may not be).
Working in $V[G]$, for every $\alpha<\beta<\omega_1$ and $n,m$, let $p_{\alpha,\beta, m,n}\in G$ decide the statement
$\langle\alpha,m\rangle<_{\dot{T}}\langle \beta, n\rangle$. Let $d_{\alpha,\beta, m,n}=\dom(p_{\alpha,\beta, m,n})$; this is a
finite  subset of $\aleph_{\omega_1+1}\times\omega$. Now let $d^*=\bigcup_{\alpha,\beta, m,n}d_{\alpha,\beta, m,n}$ and
$d=\{i<\aleph_{\omega_1+1}\mid (\exists k)\langle i,k\rangle\in d^*\}$. Then $d$ has size at most $\omega_1$. By increasing $d$
if necessary, assume that $|d|=\omega_1$.

Write $G$ as $\prod_{i\in\aleph_{\omega_1+1}}G_i$, where every $G_i$ is $Add(\omega,1)$-generic, and let $\bar{G}=\prod_{i\in
d}G_i$.  Then $\bar{G}$ is $\C_{\omega_1}$-generic, containing every $p_{\alpha,\beta, m,n}$. And so $T\in V[\bar{G}]\subset
V[G]$.

\end{proof}
Fix $\bar{G}$ as in the above lemma.
\begin{lemma} All $\omega_1$-branches of $T$ in $V[G]$ are already in the forcing extension $V[\bar{G}]$.
\end{lemma}
\begin{proof}
This is because $\C/\bar{G}$, i.e. the forcing to get from $V[\bar{G}]$ to $V[G]$ is Knaster:

Suppose $\dot{b}$ is forced to be an $\omega_1$-branch of $T$. For $\alpha<\omega$ let $p_\alpha\Vdash_{\C/\bar{G}}
u_\alpha\in\dot{b}\cap T_\alpha$,  where $T_\alpha$ is the $\alpha$-th level. Then there is an unbounded $I\subset \omega_1$ such
that for all $\alpha,\beta\in I$, $p_\alpha$ and $p_\beta$ are compatible, and so $\{u_\alpha\mid\alpha\in I\}$ generate the
branch in $V[\bar{G}]$.

\end{proof}
For every $\alpha<\aleph_{\omega_1}$, let $\P_\alpha:=\{(t,f\upharpoonright{\aleph_\alpha})\mid (t,f)\in\P\}$. Then clearly the
poset $\P$ is the
union of the sequence $\langle\P_\alpha \mid\alpha<\omega_1\rangle$ and each $\P_\alpha$ is a regular $\aleph_2$-cc
subordering of $\P$  that adds $\aleph_\alpha$ many branches to the generic Kurepa tree. Let $H_\alpha$ be the generic  filter for
$\mathbb{P}_\alpha$ obtained from $H$. Also, for a condition $p\in\mathbb{P}$, we use the notation $p=(t^p, f^p)$.

\begin{claim}\label{claim1}
For every $\omega_1$ branch $b$ of $T$, there is some $\alpha<\omega_1$, such that $b\in V_0[H_\alpha\times\bar{G}]$.
\end{claim}

\begin{proof}
As before, we index the nodes of $T$ by $\langle\alpha, n\rangle$, for
$\alpha<\omega_1$ and $n<\omega$, where the first coordinate denotes the level of the node.  Similarly to the arguments in Lemma
\ref{submodel}, we can find a set $d\subset\aleph_{\omega_1}$ of size $\omega_1$, such that $T$ is in $V_0[\bar{G}\times
\bar{H}]$, where $\bar{H}=\{(t,f\upharpoonright d)\mid (t,f)\in H\}$. Then $\bar{H}$ is actually a generic filter for
$\mathbb{P}_1$, and we view $V_0[\bar{G}\times H]=V_0[\bar{G}\times \bar{H}][H']$, where $H'$ is
$\mathbb{P}/\bar{H}:=\{(t,f)\in\mathbb{P}\mid (t,f\upharpoonright\aleph_1)\in \bar{H}\}$ -generic.

Suppose $\dot{b}$ is a $\mathbb{P}$-name for a cofinal branch through $T$, which is not in  $V_0[\bar{G}\times \bar{H}]$. We say
$p\Vdash \dot{b}(\alpha)=n$ to mean that $p$ forces that $\dot{b}\cap \dot{T}_\alpha=\{\langle\alpha,n\rangle\}$. In
$V_0[\bar{G}\times \bar{H}]$ let $\mathbb{Q}:=\{(t,f\upharpoonright \aleph_1+1)\mid (t,f)\in\mathbb{P}/\bar{H}\}$. Define a
$\mathbb{Q}$-name $\tau$, by setting
$$\tau=\{(\langle\alpha,n\rangle, q)\mid \alpha<\omega_1, n<\omega, f^q(\omega_1)=n,   \exists p\in\mathbb{P}/\bar{H},
t^p\upharpoonright (\alpha+1)=t^q, p\Vdash \dot{b}(\alpha)=n\}.$$

Also let $K=\{(t,f)\mid \exists\alpha<\omega_1, \exists p\in H, t^p\upharpoonright\alpha+1=t,  \omega_1\in\dom(f), p\Vdash
\dot{b}(\alpha)=f(\omega_1)\}$. Since $b$ is not in $V_0[\bar{G}\times \bar{H}]$, it is straightforward to check that $K$ is
$\mathbb{Q}$ -generic, and also that $\dot{b}_H=\tau_{K}$.

Finally, since $\bar{H}* K\in V_0[\bar{G}\times H]$ is generic for the suborder  $\{(t,f\upharpoonright\aleph_1+1)\mid
(t,f)\in\mathbb{P}\}$, $K$ must be in $V_0[\bar{G}\times H_\alpha]$, for some $\alpha<\omega_1$. Then $b\in V_0[\bar{G}\times
H_\alpha]$.

\end{proof}

\begin{lemma} If $\P \times \C_{\omega_1}$ adds more than $\aleph_{\omega_1}$-many $\omega_1$-branches to $T$
then there is some $\alpha<\omega_1$, so that $\P_\alpha \times \C_{\omega_1}$ adds
more than $\aleph_{\omega_1}$ many $\omega_1$-branches to $T$.
\end{lemma}
\begin{proof}
For every $\alpha<\omega_1$, let $H_\alpha$ be $\P_\alpha$-generic over $V_0$, induced by $H$. Then  $V_0[H_\alpha\times
\bar{G}]\subset V_0[H\times \bar{G}]=V[\bar{G}]$.

Suppose that for some $\lambda>\aleph_{\omega_1}$, $T$ has $\lambda$-many branches,  enumerate them by  $\langle b_i\mid
i<\lambda\rangle$.
For every $i<\lambda$, let $\alpha_i<\omega_1$ be such that $b_i\in V_0[H_{\alpha_i}\times\bar{G}]$ given by Claim \ref{claim1}.
Then for some $\alpha<\omega_1$ there is an unbounded
$I\subset\lambda$, such that for all $i\in I$, $\alpha_i=\alpha$.

But that implies that in $V_0[H_\alpha\times\bar{G}]$, $2^{\omega_1}>\aleph_{\omega_1}$, which is a contradiction since we
started with $V_0\models GCH$.
So the forcing extension of
$\P \times \C$ has at most $\aleph_{\omega_1}$ many $\omega_1$-branches of $T$, i.e
$\B=\aleph_{\omega_1}$ in this forcing extension.
\end{proof}

$V[G]=V_0[H][G]$ proves part (2) of Theorem \ref{constpart1}.

\begin{theorem}\label{constpart2} From a Mahlo cardinal, it is consistent with ZFC that $2^{\aleph_0}<\B=2^{\aleph_1}$, for every 
$\kappa<\twoalephone$ there is a Kurepa tree with (at least) $\kappa$-many maximal branches, but no Kurepa tree has 
$\twoalephone$-many maximal branches.

\end{theorem}

\begin{proof}
The proof uses the  forcing axiom principle, $GMA$, defined by Shelah and a maximality principle from \cite{LuckeMaximalityPrinciples} that generalizes $GMA$.

We start with some definitions. Let $\kappa$ be a regular cardinal. A poset $\mathbb{P}$ is  {\it stationary $\kappa^+$-linked}, if for every sequence of conditions $\langle p_\gamma\mid \gamma<\kappa^+\rangle$, there is a regressive function $f:\kappa^+\rightarrow \kappa^+$, such that for some club $C\subset \kappa^+$, for all $\gamma,\delta\in C$ with cofinality $\kappa$, $f(\gamma)=f(\delta)$ implies $p_\gamma$ and $p_\delta$ are compatible.  Let $\Gamma_{\kappa}$ be the class of $\kappa$-closed, stationary $\kappa^+$-linked, well met poset $\mathbb{P}$ with greatest lower bounds.

\begin{definition}
For a regular $\kappa$, $GMA_\kappa$ states that for every $\mathbb{P}\in \Gamma_\kappa$, and for every collection of less than $2^\kappa$ many dense sets there is a filter for $\mathbb{P}$ meeting them.
\end{definition}

\begin{definition}
For a regular $\kappa$, $SMP_n(\kappa)$ says the following: 
\begin{itemize}
\item
$\kappa^{<\kappa}=\kappa$; 
\item
for any $\Sigma_n$ statement $\phi$  with parameters in $H(2^\kappa)$ and any $\mathbb{P}\in \Gamma_\kappa$, if for all  $\mathbb{P}$-generic $G$, and all $\kappa$ closed, $\kappa^+$-c.c. posets $\mathbb{Q}\in V[G]$, we have that $V[G]\models 1_{\mathbb{Q}}\Vdash \phi$, then $\phi$ is true in $V$.
\end{itemize}

$SMP(\kappa)$ is the statement that $SMP_n(\kappa)$ holds for all $n$. 
\end{definition}
\begin{fact}
(\cite{LuckeMaximalityPrinciples}) If $\kappa$ satisfies $\kappa=\kappa^{<\kappa}$, then a model of $SMP(\kappa)$ can be forced 
starting from a Mahlo cardinal $\theta>\kappa$. 
\end{fact}

The following is also due to \cite{LuckeMaximalityPrinciples}. We include the proof for completeness.
\begin{prop}
Let $\kappa$ be a regular cardinal.
\begin{enumerate}
\item
If for all $\tau<2^\kappa$, $\tau^{<\kappa}<2^\kappa$, then $SMP_1(\kappa)$ implies that $GMA_\kappa$.
\item
$SMP_2(\kappa)$ implies that $2^\kappa$ is weakly inaccessible and that for all $\tau<2^\kappa$, $\tau^{<\kappa}<2^\kappa$.
\item
$SMP_2(\kappa)$ implies that every $\boldsymbol{\Sigma_1^1}$-subset of $\kappa^{\kappa}$ of cardinality $2^{\kappa}$ contains a 
perfect set. \footnote{Recall that $A\subset \kappa^\kappa$ contains a perfect set if there is a continuous injection 
$2^\kappa\rightarrow \kappa^\kappa$ with range contained in $A$. We assume the discrete topology on $2$ and $\kappa$ 
respectively.}
\end{enumerate}

\end{prop}
\begin{proof}

For the first item, suppose that $\tau<2^\kappa\rightarrow\tau^{<\kappa}<2^\kappa$ and $SMP_1(\kappa)$ holds. Let $\mathbb{P}\in\Gamma_\kappa$ and let $\mathcal{D}$ be a a family of less than $2^\kappa$ many dense sets. By our cardinal arithmetic assumptions, we may assume that the underlying set of $\mathbb{P}$ is an ordinal of size less than $2^\kappa$. Let $\phi$ be the statement that there exists a filter for $\mathbb{P}$ meeting every set in $\mathcal{D}$. Note that $\phi$ is $\Sigma_1$ and uses the parameters $\mathbb{P}$ and $\mathcal{D}$, which are in $H(2^\kappa)$. If $G$ is a $\mathbb{P}$-generic filter, then  $\phi$ will hold in any forcing extension of $V[G]$, so by $SMP_1(\kappa)$, we have that $\phi$ is true in $V$.

For the second item, suppose that $SMP_2(\kappa)$ holds, and let $\tau<2^\kappa$. Let $\phi$ be the statement that there is no surjection from $\tau^{<\kappa}$ to $\mathcal{P}(\kappa)$. Then $\phi$ is $\Sigma_2$ with parameters in $H(2^\kappa)$. Denote $\lambda:=\tau^{<\kappa}$. Let $\mathbb{\dot{Q}}$ be an $Add(\kappa, \lambda^+)$-name for a $\kappa$-closed forcing with $\kappa^+$ c.c. Note that since this forcing is $\kappa$-closed, $\tau^{<\kappa}$ is computed the same as in $V$. Then $\phi$ is true after forcing by $Add(\kappa, \lambda^+)*\mathbb{\dot{Q}}$.  By $SMP_2(\kappa)$ ,  $\phi$ is true in $V$. I.e. $\tau^{<\kappa}<2^\kappa$.

Next, we show that $2^\kappa$ is a limit cardinal. Suppose that $\kappa\leq \mu<2^\kappa$, and let $\psi$ be the statement $\exists x (\mu<|x|<2^\kappa)$. This is $\Sigma_2$ with parameters in $H(2^\kappa)$. Let $\nu:=2^\kappa$. Then if  $\mathbb{\dot{Q}}$ is an $Add(\kappa, \nu^+)$-name for a $\kappa$-closed forcing with $\kappa^+$ c.c., after forcing with 
$Add(\kappa, \nu^+)*\mathbb{\dot{Q}}$, $\psi$ holds as witnessed by taking $x$ to be $\nu$. So by $SMP_2(\kappa)$ ,  $\psi$ is true in $V$. So, $2^\kappa$ is a limit cardinal.

By a similar argument, $SMP_2(\kappa)$ implies that for all $\kappa\leq\mu<2^\kappa$, $2^\mu=2^\kappa$. Therefore $2^\kappa$ is regular (since by standard cardinal arithmetic if $\theta$ is singular and $2^\mu$ is constant for all large $\mu<\theta$, then $2^\theta=2^\mu$ for all large $\mu<\theta$). 

Finally, we show the third item. Again, suppose that $SMP_2(\kappa)$ holds and $T$ is a tree of $\kappa^{<\kappa}\times\kappa^{<\kappa}$. We want to show that $p[T]$ contains a prefect set or has cardinality less than $2^\kappa$. \footnote{Recall that $A\subset\kappa^\kappa$ is $\boldsymbol{\Sigma_1^1}$ iff $A=p[T]$ for some tree $T\subset\kappa^{<\kappa}\times\kappa^{<\kappa}$.}  As above, denote $\nu:=2^\kappa$. Let $\mathbb{\dot{Q}}$ be an $Add(\kappa, \nu^+)$-name for a $\kappa$-closed forcing with $\kappa^+$ c.c., and let $G*H$ be generic for $Add(\kappa, \nu^+)*\mathbb{\dot{Q}}$. Note that $V$ and $V[G][H]$ have the same cardinals. We have two cases:
\begin{enumerate}
\item
$(p[T])^V\subsetneq p([T])^{V[G][H]}$,
\item
$V[G][H]\models |p[T]|<2^\kappa$.
\end{enumerate}

Suppose the first case holds. Then by Lemma 7.5 in \cite{LuckeDefinability}, one can construct a continuous order-embedding $i:2^{<\kappa}\rightarrow \kappa^{<\kappa}\times\kappa^{<\kappa}$ with range contained in $T$, which witnesses that the projection of all cofinal branches, $p[T]$, contains a perfect set. The idea is to take a new branch that is forced not to belong to $V$ and at the inductive step, pick conditions that force splitting along that branch; the witnessing splitting nodes are used to define $i$, so that if $s,s'\in 2^{<\kappa}$ are incompatible, then so are $i(s)$ and $i(s')$. For the full details, see Lemma 7.5 in \cite{LuckeDefinability}.

It follows that the statement that there is such an embedding or $|p[T]|<2^\kappa$ is true in $V[G][H]$. This is $\Sigma_2$ with parameters in $H(2^\kappa)$ (since $\kappa^{<\kappa}=\kappa$), so by $SMP_2(\kappa)$, it is true in $V$. This concludes the proof of the last item.

\end{proof}

\begin{remark}
Regarding the first item of the above proposition, actually the converse also holds. I.e. assuming for all $\tau<2^\kappa$, $\tau^{<\kappa}<2^\kappa$, we have $SMP_1(\kappa)$ if and only if $GMA_\kappa$ and $\kappa^{<\kappa}=\kappa$. 
\end{remark}
Now, towards the proof of Theorem \ref{constpart2}, let $V$ be a model of $SMP_2(\omega_1)$. By the above proposition, in $V$ we have:
\begin{enumerate}
\item
$GMA_{\omega_1}$,
\item
$CH$,
\item
$2^{\omega_1}$ is weakly inaccessible,
\item
every $\boldsymbol{\Sigma_1^1}$-subset of $\omega_1^{\omega_1}$ of cardinality $2^{\omega_1}$ contains a perfect set.
\end{enumerate}
We claim that this is the desired model. Let $\omega_1<\lambda<2^{\omega_1}$. To see that there is a Kurepa tree with (at least) 
$\lambda$-many maximal branches, let $\mathbb{P}$ be the standard poset to add such a tree (i.e. we take the poset from earlier 
but with $\lambda$ in place of $\aleph_{\omega_1}$). Namely, conditions are pairs $(t,f)$, where:
\begin{itemize}
\item
$t$ is a tree of height $\beta+1$ for some $\beta<\omega_1$ and countable levels;
\item
$f$ is a function with $\dom(f)\subset \lambda$, $|\dom(f)|=\omega$, and $\ran(f)= t_\beta$, where $t_\beta$ is the
$\beta$-th level of $t$.
\end{itemize}
The order is the same as in the poset from Theorem \ref{constpart1}.

Then $\mathbb{P}$ satisfies the hypothesis of $GMA_{\omega_1}$\footnote{Note that by CH there are only $\aleph_1$-many 
possibilities for the first coordinate and we can apply a $\Delta$ -system lemma for the second 
coordinate to prove that $\mathbb{P}$ is stationary $\kappa^+$-linked.}, and there are only $\lambda$-many dense sets to 
meet in order to get a Kurepa tree with (at least) $\lambda$-many branches. So, by $GMA_{\omega_1}$, there is a Kurepa tree with 
(at least) $\lambda$ many branches.

Finally, we use the last item of the properties listed above to argue that there are no Kurepa trees with $2^{\omega_1}$-many branches. Suppose that $T$ is a Kurepa tree. We claim that the set of branches $[T]$ is a closed set that does not contain a perfect set. For suppose that $g:2^{\omega_1}\rightarrow {\omega_1}^{\omega_1}$ is a continuous injection with range contained in $[T]$. Then, build  sequences $\langle p_s\mid s\in 2^{<\omega}\rangle$ and $\langle\alpha_n\mid n<\omega\rangle$, such that:
\begin{itemize}
\item
if $s'\supset s$, then $p_{s'}<_T p_s$,
\item
if $|s|=n$, then $\alpha_n=\dom p_s$,
\item
for each $s$, $p_{s\smallfrown 0}\neq p_{s\smallfrown 1}$.
\end{itemize} 
We do this by induction on $|s|$, using continuity. Then let $\alpha=\sup_n\alpha_n$ and for all $\eta\in 2^\omega$, 
let $p_\eta:=\bigcup_n p_{\eta\upharpoonright n}$. But then if $\eta\neq\delta$, $p_\eta\neq p_\delta$. 
So, the $\alpha$-th level of the tree has $2^\omega$ many elements. Contradiction with $T$ being a Kurepa tree. So $[T]$ does not 
contain a perfect set. By our assumption in item (4), it follows that $|[T]|<2^{\omega_1}$.
\end{proof}

\begin{remarks}
\begin{enumerate}[(i)]
 \item There is a similar argument to the above proof in Section IV.2 in \cite{syandfriends}. The idea of using Kurepa 
tree 
to get counterexamples of the perfect set property goes back to \cite{MeklerVaananen}.
\item For simplicity we used a Mahlo cardinal, which suffices to provide us a model of $SMP_2(\omega_1)$. A slightly weaker 
hypothesis would be enough. For precise upper and lower consistency strength bounds on $SMP(\omega_1)$, see Theorem 4.5 in 
\cite{LuckeMaximalityPrinciples}.
\end{enumerate}
\end{remarks}

Finally we want to note that the results presented here can be extended to
$\kappa$-Kurepa trees with $\kappa\ge\aleph_2$.

\subsection*{Acknowledgements} The ideas of Lemma \ref{bknotmax} and Theorem \ref{constpart2} were communicated to the
authors by Philipp L\"{u}cke. The consistency results of Section \ref{Consistency} were based on early ideas of Stevo Todorcevic.
The authors would also like to thank John Baldwin for his useful feedback on an earlier version of the manuscript that
helped improve the exposition.

\bibliographystyle{plain}   

\bibliography{AllBibliography.bib}

\end{document}